\documentclass[12pt]{article}           
\usepackage{amssymb}
\usepackage{amsmath}    
\usepackage{enumerate}  
\usepackage{color}

\newtheorem{thm}{Theorem}[section]

\newtheorem{prop}[thm]{Proposition}

\newtheorem{defi}[thm]{Definition}
\newtheorem{example}[thm]{Example}
\newtheorem{notation}[thm]{Notation}
\newtheorem{note}[thm]{Note}

\numberwithin{equation}{thm} 

\newenvironment{proof}[1][Proof]{\begin{trivlist}
\item[\hskip \labelsep {\bfseries #1}]}{\end{trivlist}}

\newcommand{\qed}{\nobreak \ifvmode \relax \else
      \ifdim\lastskip<1.5em \hskip-\lastskip
      \hskip1.5em plus0em minus0.5em \fi \nobreak
      \vrule height0.75em width0.5em depth0.25em\fi}

\begin{document}


\title{Randomly Perturbed Ergodic Averages \\ {\small JaeYong Choi and   Karin Reinhold\\    
Department of Mathematics and Statistics\\
University at Albany, SUNY \\ 
}}
\maketitle







\begin{abstract}   Convergence properties of random ergodic averages have been extensively studied in the literature. In these notes, we exploit a uniform estimate by Cohen \& Cuny who showed convergence of a series along randomly perturbed times for functions in $L^2$ with $\int \max(1,\log (1+|t|)) d\mu_f<\infty$.  We prove universal pointwise convergence of a class of random averages along randomly perturbed times for $L^2$ functions with  $\int \max(1,\log\log(1+|t|)) d\mu_f<\infty$. For averages with additional smoothing properties, we obtain a universal variational inequality as well as universal pointwise convergence of a series define by them for all functions in $L^2$. 
\end{abstract}

\section{Introduction.}

Ergodic averages along sequences that are well behaved, in terms of pointwise convergence, loose the convergence properties if the sequence is slightly perturbed. To be more specific, a sequence of operators $T_n$ acting on a probability space has the {\it strong sweeping out  property} if, for any $\epsilon>0$, there exists a set $E$ with $0<m(E)<\epsilon$, such that $\limsup_{n\to\infty} T_n 1_E =1 $ a.e. and $\liminf_{n\to\infty} T_n 1_E =0 $ a.e..  Let $(X,{\cal D}, m)$ be a non--atomic probability space and 
$\{\tau_t\}_{t\in \mathbb R}$ an aperiodic, ergodic, measure preserving flow on it. 
Ackoglu, Bellow, del Junco and Jones \cite{ABDJ} showed that for any increasing sequence of integers $\{n_k\}$, if $\{t_k\}$ is a sequence such that $t_k\mapsto 0$, the averages
\begin{equation*}
 B_nf(x) = \frac{1}{n} \sum_{k=1}^n   f( \tau_{n_k+t_k} x)   
\end{equation*}
have the strong sweeping out property and therefore, there exist bounded functions for which pointwise convergence fails on sets of positive measure. 
Earlier
Bergelson, Boshernitzan and Bourgain  \cite{BBB} had used Bourgain's entropy method \cite{B} to show that the averages $B_nf$ diverge a.e. for some $f\in L^{\infty}$ when the sequence $\{t_k\}$ is independent over the rationals.
Ackoglu, del Junco and Lee \cite{ADL} proved that the related averages 
\begin{equation*}  
C_nf(x)=\frac 1n \sum_{k=1}^n   f( \{x+t_k\} ) \, dt 
\end{equation*}    
 for $f$ in $[0,1)$, have the $\delta$-sweeping out property.
 In \cite{ABJLRW}, Ackoglu et all showed that these averages have in fact the strong sweeping out property.

In these notes we consider the behavior of averages such as $B_nf$ and $C_nf$ 
when the perturbations $\{t_k\}$ are random rather than deterministic. 
For example, let $\{n_k\}$ be a non-decreasing sequence of integers, 
$\{\delta_k\}$ and $\{\epsilon_k\}$ independent sequences of i.i.d. random variables 
defined on a probability space $(\Omega,{\cal F},{\bf P})$. 
Given a semi--flow of positive contraction operators 
$\{T_t\}_{t\in ({\mathbb R}^+)}\subset ({\mathbb R}^+)$, 
we may consider the following averages:
\begin{equation} F_nf(\omega,x)=\frac 1n \sum_{k=1}^n \frac 1{2|\epsilon_k(\omega)|} 
\int_{|t|<|\epsilon_k(\omega)|} T_{n_k+\delta_k(\omega)+t}f(x) \, dt,  \label{eq:intro} 
\end{equation} 
\begin{equation*}  
G_nf(\omega,x)=\frac 1n \sum_{k=1}^n   T_{n_k+\delta_k(\omega)}f(x)  
\quad \mbox{ and } \quad 
H_nf(\omega,x)=\frac 1n \sum_{k=1}^n   T_{\delta_k(\omega)}f(x). \label{Hn} 
 \end{equation*}

``Universal'' convergence results concerns finding a subset $\Omega'\subset \Omega$ with 
$P(\Omega')=1$ such that, for every $\omega\in\Omega'$, the averages in consideration converge 
almost everywhere for all functions in a certain class.
Reinhold \cite{KR} considered averages $F_nf$ for the sequence $n_k=k$ and obtained universal
convergence results for $f\in L^p$, $p\ge 1$, when $e_k^{-1}\in L^q$, $\frac 1p +\frac 1q$. 
Schneider \cite{Sch1} considered averages of the form $G_nf$ for the sequence $n_k=k^2$  
and the sequences $\delta_k$ taking values $1$ and $-1$ with probability 1/2.  
He established universal a.e. converge for $f\in L^2$. Schneider in \cite{Sch2} 
then extended this result for sequences with the growth condition $n_k=O(2^{k^s})$ 
for some $s\in (0,1)$, and integer--valued sequences of random variables $\delta_k$. 
Cohen and Cuny \cite{Cohen.Cuny2} considered the averages $G_nf$ and $H_nf$ for sequences 
$\{\delta_k\}$ that do not take integer values and $f\in L^2$. 
In light of \cite{BBB,ABJLRW}, pointwise convergence of $G_nf$ fails for some functions in $L^2$ 
but positive results are obtained by considering a subclass. 
Cohen and Cuny \cite{Cohen.Cuny2} proved universal pointwise convergence of averages 
$G_nf$ and $H_nf$ for $f\in L^2$  such that $\int \log(2+|t|) d\mu_f(t)<\infty$. 

\begin{thm} [\cite{Cohen.Cuny2}, Theorem 4.12] Let $\{X_n\} \subset ({\mathbb R}^+)^d$ be 
i.i.d random variables with $E(|X_1|^{\alpha})<\infty$ for some $\alpha>0$. 
Let $\{n_k\} \subset ({\mathbb R}^+)^d$, with $|n_m|^*=max_{k\le m} |n_k|= O(2^{m^\beta})$, 
for some $0<\beta<1$. Then, there exists a set $\Omega'\subset \Omega$ with $P(\Omega')=1$, 
such that for every $\omega\in \Omega'$, for any probability space $(X,{\cal D},m)$, 
any continuous semi-flow of isometries $\{T_t\}_{t\in ({\mathbb R}^+)^d}$, 
and any $f\in L^2(X)$ with $\int_{{\mathbb R}^d} \log(2+|t|) d\mu_f<\infty$, the series 
\[\sum_{k=1}^{\infty} \frac{T_{n_k+X_k(\omega)}f -T_{n_k}E(T_{X_k})f}{k}\]
converges almost everywhere. In particular 
\[ \frac 1n \sum_{k=1}^{n} (T_{n_k+X_k(\omega)}f -T_{n_k}E(T_{X_k})f)\]
converges to $0$ almost everywhere.
\end{thm}

Cohen and Cuny's approach aimed at proving the convergence of the above series which yielded pointwise convergence as a consequence.
In section 3, we exploit the uniform estimates they obtained for almost periodic polynomials to prove a variational inequality. This approach allowed us to obtain universal pointwise convergence of the above averages under the weaker condition $\int \max(1,\log  \log(|t|)) d\mu_f(t)<\infty$. 

\begin{thm} \label{thm:1} Let $\{\delta_n\}_{n\in \mathbb N} \subset ({\mathbb R}^+)^d$ be independent random vectors with $\sum_{k\ge 1}  P(|\delta_k|>2^{k^{\beta}})<\infty$.  Let $\{n_k\} \subset ({\mathbb R}^+)^d$, with $ |n_k|= O(2^{k^{\beta}})$, for some $0<\beta<1$. Then, there exists a set $\Omega'\subset \Omega$ with $P(\Omega')=1$, such that for every $\omega\in \Omega'$, for any probability space $(X,{\cal D},m)$, any continuous semi-flow of positive isometries $\{T_t\}_{t\in ({\mathbb R}^+)^d}$, and any $f\in L^2(X)$ with $\int_{{\mathbb R}^d} \max(1,\log\log(|t|)) d\mu_f<\infty$, the averages
\[ \frac 1{\lfloor \rho^n \rfloor} \sum_{k=1}^{\lfloor \rho^n \rfloor} (T_{n_k+\delta_k(\omega)}f -T_{n_k}E(T_{\delta_k})f)
\]
converges to $0$ almost everywhere for any $\rho>1$. 
If in addition the sequence $\{\delta_k\}$ are i.i.d. and the sequence $\{n_k\}$ is such that the regular ergodic averages along that sequence satisfy a variational inequality $\| \| \frac 1{\lfloor \rho^n \rfloor} \sum_{k=1}^{\lfloor \rho^n \rfloor} T_{n_k}f\|_{v(s)} \|_2 \le C \|f\|_2$ for any $\rho>1$ ($s>2$), then the averages 
$ G_nf(\omega,x)= \frac 1n \sum_{k=1}^{n} T_{n_k+\delta_k(\omega)}f \mbox{ converge a.e.}. $
\end{thm}

Note that when $n_k=0$ for all $k\ge 1$, and $\{\delta_k\}$ are i.i.d., the above theorem gives a universal convergence result for the averages $H_nf(\omega,x)=\frac 1n \sum_{k=1}^{n} T_{\delta_k(\omega)}f $. Theorem \ref{thm:1} is proven in section 3 for a wider class of averages that include $r$--dimensional averages of random measures.

When the kernels associated with the averages have additional smoothing properties, such as the random averages $F_nf$ above, the averages converge almost everywhere for all $f\in L^2$, provided $1/\epsilon_k\in L^{\alpha}$ for some $\alpha>0$.

Let $\zeta:\mathbb R^d \to \mathbb R$ be positive, integrable with $\int_{\mathbb R^d} \zeta(t) dt =1$. Let $L_{\epsilon}f(x) = \int_{\mathbb R^d} 
\zeta_{\epsilon}(t) T_tf(x) dt$ where $\zeta_{\epsilon}(t)=\frac 1{\epsilon_1 \cdots \epsilon_d} \zeta(\frac{t_1}{\epsilon_1},\ldots,\frac{t_d}{\epsilon_d})$. Let $\{\delta_k\}_{k\in \mathbb N^r}$ and $\{\epsilon_k\}_{k\in \mathbb N^r}$ be independent sequences of i.i.d. positive random vectors, and $\{n_k\}_{k\in \mathbb N^r} \subset (\mathbb R^+)^d$. From here on, $F_nf(x)$
denotes the following smoothed average around the observations ${{n_k}+\delta_k(\omega)}$,
\begin{equation}F_nf(\omega, x) = F_nf(x) = \frac {1}{n^r} \sum_{k\in [1,n]^r} T_{n_k+\delta_k(\omega)}L_{\epsilon_k(\omega)} f(x).
\label{eq:Fn}
\end{equation}
Additionally, for $u\in \mathbb R^d$, $|u|=\max_{1\le i\le d} |u_i|$; and for $u,s\in \mathbb R^d$, $<u,s>=\sum_{i=1}^d u_i s_i$, and $u.s=(u_1s_1,\ldots,u_ds_d)$.

\begin{prop} \label{thm:2} 
Let $\{\delta_k\}_{k\in \mathbb N^r}$ and $\{\epsilon_k\}_{k\in \mathbb N^r}$ be two independent sequences of i.i.d positive random vectors in $\mathbb R^d$, and $\{n_k\}_{k\in \mathbb N^r} \subset (\mathbb R^+)^d$.
Assume they satisfy the following conditions:
\begin{enumerate} \setlength{\itemsep}{\smallskipamount}   
\item[a.] $E( \min_{1\le j\le d} |\epsilon_{e,j}|^{-\alpha})<\infty$ for some $\alpha>0$; 
\item[b.] $\sum_{j\ge 1} j^{r-1} P(\delta_{e}>2^{j^{r\beta}})<\infty$,   for any $e\in \mathbb N^r$;  and
\item[c.]$|n_k|=O(2^{|k|^{r\beta}})$, for some $0<\beta<1$. 
\end{enumerate}

Then, there exists a set $\Omega'\subset \Omega$ with $P(\Omega')=1$, such that for every $\omega\in \Omega'$, for any probability space $(X,{\cal D},m)$, and any continuous semi-flow of positive isometries $\{T_t\}_{t\in ({\mathbb R}^+)^d}$, the partial sums
\[\sum_{k\in [1,n]^r} \frac{T_{n_k+\delta_k(\omega)} L_{\epsilon_k(\omega)}f(x) - T_{n_k}E(T_{\delta_k} L_{\epsilon_k})f(x)}{|k|^r}\]
 converges a.e. for any  $f\in L^2(X)$ and 
 \[\left \| \sum_{k\in \mathbb N^r} \frac{T_{n_k+\delta_k(\omega)} L_{\epsilon_k(\omega)}f(x) - T_{n_k}E(T_{\delta_k} L_{\epsilon_k})f(x)}{|k|^r} \right \|_2 \le C \|f\|_2.\]

If in addition $\sup_t \prod_{j=1}^d \max (1,|t_j|^{\alpha}) |\hat{\zeta}(t)|<\infty$, for some $\alpha>0$, and the averages along the subsequence $\{n_k\}$ satisfy
$\Big\| \| A_nf \|_{v(s)} \Big\|_2 \le C \|f\|_2$ ($s>2$), then 
there exists a positive function  $C(\omega)$ finite P--a.e., such that, for every $\omega\in \Omega'$, for any probability space $(X,{\cal D},m)$, and any continuous semi-flow of positive isometries $\{T_t\}_{t\in ({\mathbb R}^+)^d}$,
\[\Big\| \| F_nf(\omega,.) \|_{v(s)} \Big\|_2 \le C(\omega) c(\beta) \|f\|_2, \quad \mbox{ and}\]
 \[\lim_{n\to \infty} F_nf(\omega,x) \mbox{ exits $m$--a.e. for all }f\in L^2(X).\]
\end{prop}

\section{Uniform Estimates}
Universal pointwise convergence theorems in $L^2$ can be obtained, by using a transfer argument or spectral representation, through uniform estimates of Fourier transforms of the associated kernels. Estimates for random trigonometric polynomials have been essential in proving convergence of random Fourier series as well as ergodic averages along subsequences and modulated ergodic averages. Paley and Zygmund(1930--32) \cite{PZ} and Salem and Zygmund (1954) \cite{SZ} provided the first estimates for trigonometric sums in their study of Fourier series with random signs: $\sum_{k=1}^{\infty} \epsilon_k c_k e^{ikx}$ where the $\{\epsilon_k\}$ is a Rademacher sequence, and $\{c_k\}$ is a sequence of complex numbers.   They were also used to prove convergence of random Fourier and almost periodic series \cite{ CuzickLai,FanSchneider,Weber, Cohen.Cuny2, Cohen.Cuny3,Cohen}.

In ergodic theory, their study yielded applications to the convergence of averages along subsequences and averages with random weights. Bourgain  \cite{B3,B2,B1,B}  used them to prove pointwise convergence of ergodic averages along polynomial sequences and Bourgain and Weirdl \cite{Weirdl} applied them to pointwise convergence of ergodic averages along sequences of primes. Bourgain, Bergelson and Boshenitzan \cite{BBB} use them to prove pointwise convergence of ergodic averages with random weights as well as Assani \cite{A1, A2}, Rosenblatt and Weirdl \cite{ WirdlRos}, and Cohen and Lin \cite{CohenLin}.  Schneider \cite{Sch2} used them to prove convergence of ergodic averages along perturbed  sequences of squares, with integer perturbations. 

These works used an estimate of the associated trigonometric polynomial by means of an estimate on the derivative of those polynomials. The work of 
 Cohen and Cuny \cite{Cohen.Cuny3} extended the estimates of Salem and Zygmund  to obtain uniform estimates of multidimensional random exponential sums
of the form $\sum_{k=1}^n X_k \, e^{i(\alpha_k.t)}$, where $\{X_n\}$ is a sequence of random variables,
$\{\alpha_k\}\subset \mathbb R^d$ are sequences of real numbers, and $t\in \mathbb R^d$. The estimates on the exponential sums 
 allowed them to prove pointwise convergence results for series of the form 
 \[\sum_{k=1}^{\infty} X_k T^{n_k} f\]
 where $\{n_k\}\subset {\mathbb N}^d$, $T^{n_k}f=T_1^{n_{k,1} }T_2^{n_{k,2}}...T_d^{n_{k,d}}f$, where $T_1,...,T_d$ are commuting isometries. Several additional authors studied the convergence of power series of contractions including  Assani \cite{A-duality}, Boukhari and Weber \cite{BW}, Cohen and Lin \cite{CohenLin}, Cohen and Cuny \cite{Cohen.Cuny2,Cohen.Cuny3} and  Cohen \cite{Cohen}.

To handle the kernels corresponding to the averages $F_nf, G_nf,H_nf$ above, we introduce a larger framework that allow us to work with $\mathbb R^d$ --actions rather than $\mathbb Z^d$--actions. 

\begin{defi}
Let $(\Omega,{\cal F},{\bf P})$ be a probability space, and $\cal B$ the Borel sigma--algebra on $\mathbb R^d$, $d\ge 1$.  A function $\nu:\Omega \times {\mathbb R^d}\mapsto {\mathbb C^d}$ is a finite complex valued {\sc transition measure} on $\Omega \times {\mathbb R^d}$ if
\begin{enumerate}[i.]
\item $\nu(\omega,.)$ is a finite complex valued measure on $\cal B$, for any $\omega \in \Omega$, and
\item $\nu(.,B)$ is an ${\cal F}$--measurable function for any $B\in {\cal B}$.
\item Letting $|\nu(\omega)|:=|\nu|(\omega,\mathbb R^d)$ denote its variation norm, we also require
 $E(|\nu( . )|)<\infty$.
\end{enumerate}
\end{defi}

\begin{notation} 
If $\nu$ is a finite complex valued transition measure, $E\nu$ denotes the measured on $\cal B$ defined by $E\nu(B) = \int_{\Omega} \nu(\omega,B) dP$, for any $B\in \cal B$.
\end{notation}

 For $\{T_t\}_{t\in (\mathbb R^+)^d}$ a semi--flow of $L^2$ contractions, convolution with measures $\nu(\omega,.)$ define bounded operators in $L^2$:
 \[ \nu^{\omega} f(x) = \int_{\mathbb R^d} T_tf(x) \nu(\omega,dt).\]

\begin{defi} Let $\{\nu_k\}_{k\in I}$ be a collection of transition measures.
The sequence $\{\nu_k\}$ is {\sc independent} if for every finite set of Borel measurable simple functions $g_1,\ldots,g_m$  on $\mathbb R^d$, and any finite set $k_1,k_2,\ldots,k_m$ of pairwise distinct indices, the random variables $\{ \int_{\mathbb R^d} g_i(x)\nu_{k_i}(\omega,dx) \}_{i=1,\ldots ,m}$  are independent.
\end{defi}

Given a sequence of independent transition measures $\{\nu_k\}_{k\in \mathbb N^r}$,
consider averages of the form
\[K^{\omega}_nf(x)=\frac 1{n^r} \sum_{k\in [1,n]^r} \nu_k^{\omega}(f)(x).\]
The kernels associated with such averages are  
\[\frac 1{n^r} \sum_{k\in [1,n]^r} \hat{\nu}^{\omega}_k(t), \mbox{ for } t\in \mathbb R^d \]
where 
$\hat{\nu}^{\omega}_k(t)=\int e^{i<t,u>} \nu(\omega,du)$
is the Fourier--Stieljes transform corresponding to the measure $\nu_k(\omega,.)$.
Uniform estimates that control such kernels have already been considered by Cohen in \cite{Cohen}.

\begin{thm} [Theorem 2.8 \cite{Cohen}] \label{thm:Main} Let $\{L_k\}$ be a sequence of positive numbers, $L_k\ge 1$, such that $\sum_{n=1}^{\infty}\sum_{m=n+1}^{\infty} (1/L_{n,m}^2) <\infty$, with $L_{n,m}=\sum_{k=n+1}^m L^2_k$. 
Let $\{\nu_k\}$ be a sequence of independent, finite complex valued transition measure on $\Omega\times {\cal B}$ with $ \| \, |\nu_k|(.,\mathbb R^d) \, \|_{L^{\infty}(\Omega)} <\infty$, for all $k\ge 1$. 
Let \[P_{n,m}(t)=\sum_{k=n+1}^m \left[\int_{[-L_k,L_k]^d} e^{i<t,u>} \nu_k(du)-\int_{[-L_k,L_k]^d} e^{i<t,u>} E\nu_k(du)\right] \] 
be the sum of the difference of the (truncated) Fourier--Stieljes transform corresponding to  the measures $\{\nu_k( . ,t)\}$ and their expected values. 
And let  
$V_{n,m}=\sum_{k=n+1}^m \| |\nu_k| \|_{L^{\infty}(\Omega)}^2$. 
Then, there exists $\epsilon>0$ and $C>0$, independent of $\{\nu_k\}$, such that
\[  \left \| \sup_{m>n} \sup_{T\ge 2} \exp({\epsilon \  \frac{\max_{t\in [-T,T]^d} |P_{n,m}(\omega,t)|^2}{V_{n,m} \log(L_{n,m}^{2+d/2} \, T^{d+2})}}) \right \|_{L^1(\Omega)}<C.\] 
\end{thm}

An immediate consequence is the following application. 
\begin{prop} \label{prop:Main} Let $\varphi:\mathbb R \to \mathbb R^+$ with $\varphi(x)\gtrsim |x|$. 
Let $\{\nu_k\}_{k\in \mathbb N^r}$ be independent complex valued transition measures on 
$\Omega\times {\cal B}$, 
$\{a_k\}_{k\in \mathbb N^r}$ be a sequence in $(0,1]$.   
Assume that 
\[\sum_{k\in \mathbb N^r} a_k E|\nu_k|([|t|>\varphi(|k|)]) < \infty.\]
Let $N_{n,m}=\{k\in \mathbb N^r: n<|k|\le m\}$.
Then 
there exists $C:\Omega\to \mathbb R^+$ finite P--a.e. such that, for all $m$, 
\begin{equation} \label{eq:main} \sup_{m>n} \sup_{T>2} \frac  
{\max_{t\in [-T,T]^d}  \left | \sum_{k\in N_{n,m}} a_k ( \hat{\nu}^{\omega}_k(t )-\hat{E\nu}_k(t))  \right |^2}  
{ 1+
\sum_{k\in N_{n,m}} a_k^2  \left \| |\nu_k| \right \|^2_{L^{\infty}(\Omega)} \log(\max(\varphi(m), T))} 
\le c(d) C(w). 
\end{equation}
In particular, if $\{\nu_k\}_{k\in \mathbb N^r}$ are independent probability transition measures with 
\[\sum_{k\in \mathbb N^r} E\nu_k([|t|>\varphi(|k|)]) < \infty, \mbox{ then}\]
\begin{equation} \label{eq:main2}   
\sup_{m>n} \sup_{T>2} \frac 
{\max_{t\in [-T,T]^d} \left | \sum_{k\in N_{n,m}}     
a_k ( \hat{\nu}^{\omega}_k(t )-\hat{E\nu}_k(t))  \right |^2 }
{\left [ \sum_{k\in N_{n,m}} a_k^2 \right ] \log(\max(\varphi(m), T))}
\le c(d) C(w) .
\end{equation}
\end{prop}

\begin{proof}
First note that 
$ | a_k (\hat{\nu^{\omega}}_k(t)-\hat{E\nu}_k(t))| \le 2 a_k  \left \| |\nu_k| \right \|_{L^{\infty}(\Omega)}$. Thus $V_{n,m}$ in theorem \ref{thm:Main} becomes $\sim \sum_{k\in N_{n,m}} a_k^2  \left \| |\nu_k| \right \|^2_{L^{\infty}(\Omega)}$. 
Let  $J_n=[-\varphi(n),\varphi(n)]^d$. 
With this choice, the numbers $L_n= \varphi(n)$ satisfy the condition of Theorem \ref{thm:Main}. Let $\hat{D\nu}_k(\omega,t)= a_k (\hat{\nu^{\omega}}_k(t)-\hat{E\nu}_k(t))$.

Decompose $\sum_{k\in N_{n,m}}  \hat{D\nu}_k(\omega,t) =P_{n,m} + Q_{n,m},$
where
\[
  P_{n,m}(\omega,t)=  \sum_{k\in N_{n,m}} a_k  \left[ \int_{J_{|k|}} e^{i<t,u>} \nu_k(\omega,du)
-  \int_{J_{|k|}} e^{i<t,u>}  E\nu_k(du) \right], \]
and
\begin{align*}
 Q_{n,m}(\omega,t) =  & 
\sum_{k\in N_{n,m}} a_k \left[ \int_{J^c_{|k|}} e^{i<t,u>} \nu_k(\omega,du)
-  \int_{J^c_{|k|}} e^{i<t,u>}  E\nu_k(du) \right] 
 \end{align*}
The estimate for $P_{n,m}$ is obtained from the previous theorem by computing 
\[L_{n,m} \sim \sum_{n<j\le m} \varphi^2(j) \lesssim m \varphi(m)^2\lesssim \varphi(m)^{3}.\] 
Yielding the existence of $C_1:\Omega\to \mathbb R^+$ finite $P$--a.e. such that
\begin{align*}
\max_{t\in [-T,T]^d}  |P_{n,m}(\omega,t)|^2 \lesssim & C_1(\omega) V_{n,m} \log(L_{n,m}^{2+d/2}T^{d+2}) \\
\sim &
c(d) C_1(\omega) \left [ \sum_{k\in N_{n,m}} a_k^2  \left \| |\nu_k| \right \|^2_{L^{\infty}(\Omega)}   \right ] \log(\max(\varphi(m), T)).
\end{align*}
We also have
\[|Q_{n,m}(\omega,t)| \lesssim \sum_{k\in N_{n,m}} a_k [|nu_k|(\omega,J^c_{|k|}) + E|\nu_k|(J^c_{|k|})]
.\]
By assumption, there exists $C_2(\omega)\in L^1(\Omega)$ such that 
\[ \sup_{m>n\ge 1} \sup_t |Q_{n,m}(\omega,t)| \lesssim C_2(\omega).\]   
Combining both estimates, (\ref{eq:main}) is obtained.

If $\{\nu_k\}_{k\in \mathbb N^r}$ are independent probability transition measurThus, 
then 
\[|Q_{n,m}(\omega,t)| 
\lesssim \left [ \sum_{k\in N_{n,m}} a_k^2 \right ]^{1/2} \left [ \sum_{k\in N_{n,m}} 
[\nu_k(\omega,J^c_{|k|}) + E\nu_k(J^c_{|k|})] \right ]^{1/2}
.\]
Thus, if $\sum_{k\in \mathbb N^r} E\nu_k([|t|>\varphi(|k|)]) < \infty$, then 
there exists $C_3(\omega)\in L^1(\Omega)$ such that 
\[  \sup_t |Q_{n,m}(\omega,t)|^2 \le C_3(\omega) \left [ \sum_{k\in N_{n,m}} a_k^2 \right ] \quad \mbox{for all } m>n.\]   
Combining both estimates, (\ref{eq:main2}) is obtained.

\hfill $\square$
\end{proof}

\begin{example}\label{ex:2} Let $\{n_k\}_{k\in \mathbb N^r}\subset ({\mathbb R^+})^d$, $\{\delta_k\}_{k\in \mathbb N^r}$ be i.i.d. random vectors.
Let $\{\theta_k\}_{k\in \mathbb N^r}$ be independent copies of a probability measure  $\theta$ on $\mathbb R^d$ which are independent of $\{\delta_k\}$. Let $\nu_k=\theta_k * \delta_{n_k+\delta_k}$. Now 
\[P_m(\omega,t)=\sum_{k\in [1,m]^r} (\hat{\nu^{\omega}_k}(t) - \hat{E\nu}_k(t))
=\sum_{k\in [1,m]^r} e^{i<n_k,t>} (\hat{\nu}^{\omega}_k e^{i<\delta_k({\omega}),t>}-\hat{E\nu}_k E(e^{i<\delta_k,t>})).\]

 If $|n_k| \le c(2^{|k|^{r\beta}})$, for some $0<\beta<1$ and all $k$, we use $\varphi(x)=(2+c) \, 2^{|x|^{r\beta}}$, and the assumption of Proposition \ref{prop:Main} becomes
 \begin{align*}
 \sum_{k\in \mathbb N^r} \nu_k(|t|>(2+c) \ 2^{|k|^{r\beta}})
 \le & \sum_{k\in \mathbb N^r} \left [\theta(|t|> 2^{|k|^{r\beta}}) + P(\delta_e>2^{|k|^{r\beta}}) \right ] \\
 \le  & \sum_{n=1}^{\infty} n^{r-1} \left [\theta(|t|> 2^{n^{r\beta}}) + P(\delta_e>2^{n^{r\beta}}) \right ].
 \end{align*}
If this series is finite, then by Proposition \ref{prop:Main} there exists  $C(\omega)\in L^1(\Omega)$, such that
\[ \sup_{m\ge 1} \sup_{T\ge 2}  \max_{t\in [-T,T]^d} \frac{ |P_{m}(\omega,t)|^2}{m^r \log(\max(T, 2^{m^{r\beta}}))} =C(\omega)<\infty.\]
In particular, for $\hat{D}_n(\omega,t)=\frac 1{n^r} P_n(\omega,t)$ and $T\le 2^{n^{r\beta}}$, we obtain
\[  \max_{t\in [-T,T]^d} |\hat{D}_n(\omega,t)|^2 \le C(w) \frac{\log(\max(T, 2^{n^{r\beta}}))}{n^r}
\le C(w)\frac 1{n^{r(1-\beta)}}.\] 

\end{example}

\section{Applications to Ergodic Theory}
In this section,  
$\{\nu_k\}_{k\in \mathbb N^r}$ denotes a sequence of independent transition measures on $\Omega\times \cal B$, 
$V_{m}=\sum_{k\in [1,m]^r} \| |\nu_k| \|_{L^{\infty}(\Omega)}^2 \label{sec3:Vm}$ and $b_m=\sum_{k\in [1,m]^r} \| |\nu_k|(.,\mathbb R^d)\|_{L^{\infty}(\Omega)}$. 
The averages 
\[K^{\omega}_nf(x)=K_nf(x)=\frac 1{b_n} \sum_{k\in [1,n]^r} \nu^{\omega}_k(f)(x)\] 
have associated kernels
\begin{equation}
\label{eq:Kn} \hat{K}^{\omega}_n(t) =\hat{K}_n(t) =\frac 1{b_n} \sum_{k\in [1,n]^r} \hat{\nu}_k({\omega},t)     
=\frac 1{b_n} \sum_{k\in [1,n]^r}  \int e^{i<t,u>} \nu_k(\omega,du).
\end{equation} \\
Note that if the $\nu_k$'s are probability measures, then $b_m=V_m=m^r$, the number of terms of the averages in consideration.

The above averages have convergence properties provided that the averages defined by their expected values 
\[ E_n(t) = E(K_n(t)) = \frac 1{b_n} \sum_{k\in [1,n]^r} E\nu_k f(x).
\] 
are well behaved averages. Let
\[ \hat{E}_n(t) = E(\hat{K}_n(t)) = \frac 1{b_n} \sum_{k\in [1,n]^r} E(\hat{\nu}_k^{\omega})(t) 
= \frac 1{b_n} \sum_{k\in [1,n]^r}  \int e^{i<t,u>} E\nu_k(du),
\] 
and  let $\hat{D}_{n}(t)=\hat{K}_n(t)-\hat{E}_n(t) $ 
denote the corresponding centered kernels.
Lastly, for $\{n_k\}_{k\in \mathbb N^r} \subset {\mathbb R}^d$,  $A_nf=(1/n^r) \sum_{k\in [1,n]^r} T_{n_k}f(x)$ denote the regular averages along the sequence $\{n_k\}_{k\in\mathbb N^r}$.

%
%
%

\begin{note} \label{representation}	
A family of operators $\{T_t\}_{(\mathbb R^+)^d}$ acting on $L^2(X)$ is a continuous semi--flow of isometries, or a $(\mathbb R^+)^d$--action, if for all $f\in L^2$, $T_0f=f$ , $T_s T_t f = T_{s+t}f$ and $s,t \in (\mathbb R^+)^d$, and the map $t\to <T_tf,f>$ is continuous.

Spectral representation of unitary actions allow us to study random averages of the form $K_nf=\frac 1{b_n} \sum_{k\in [1,n]^r} \nu^{\omega}_kf$ by exploiting properties of the kernels $\hat{K}_n(t)=\frac 1{b_n} \sum_{k\in [1,n]^r} \hat{\nu}^{\omega}_k$. 

For unitary actions of $\mathbb R^d$,  a classical generalization of Stone's Theorem \cite{RN} provides a spectral representation. That is, 
for any $f\in L^2$, there is a positive finite measure $\mu_f$ on $\mathbb R^d$, called its spectral measure,  such that for any $s \in \mathbb R^d$, $<T_sf,f>=\int_{\mathbb R^d} e^{i<s,t>} d\mu_f(t)$.

The unitary dilation theorem implies that results obtained for unitary actions yield results for actions by isometries as well.
\begin{thm} Riesz--Nagy \cite{RN} \label{thm:dilation}
Let $\{T_t\}_{t\in (\mathbb R^+)^d}$ be a representation of  $(\mathbb R^+)^d$ by isometries on a Hilbert space H. There exist a Hilbert space $H_2 \supset H$ and an $\mathbb R^d$ action $\{U_t\}_{t\in \mathbb R^d}$  by unitary operators on $H_2$, such that, if $\pi:H_2\to H$ is the orthogonal projection, then $\pi U_t x=T_tx$, for all $t\in (\mathbb R^+)^d$ and $x\in H$.
\end{thm}

\end{note}


Define $(\log\psi)^{+}(t)=\log_2 \psi (t)$ if $\psi(t)>2$, and $1$ otherwise. For any positive real number  $v$, abusing notation $K_vf=K_{\lfloor v\rfloor }f$.

\begin{thm} \label{thm:estimate.for.D.1}
Let $\{\nu_k\}_{k\in \mathbb N^r}$ be a sequence of independent, finite complex valued transition measures on $\Omega\times {\cal B}$. 
Let $\varphi$ be an non--decreasing positive function such that $\varphi(|x|) \gtrsim |x|$, and let $\psi$ be its (generalized)  inverse. 
Assume that 
\begin{align*} 
\mbox{(a)}&\quad \sum_{k\in \mathbb N^r} E|\nu_k|(|t|>\varphi(|k|)) <\infty,  
\qquad \mbox{and} \\
\mbox{(b)}& \quad 
\sum_{n\ge 1} \frac{ 1+ V_{\rho^n}  \log( \varphi(\rho^{ n}))  }{b_{\rho^n}^2}<\infty, \mbox{ for all } \rho>1.
\label{eqn1} 
\end{align*}
Then, there exists $C(\omega)\in L^1(\Omega)$ such that
 for any probability space $(X,{\cal D},m)$ with any continuous semi--flow of isometries $\{T_t\}_{t\in \mathbb (R^+)^d}$ acting on it, if $f\in L^2(X)$, 
\[\Bigg\|  \sqrt{ \sum_{n=1}^{\infty} | D_{\rho^n}f |^2 }  \Bigg\|^2_2<  C(\omega)  \, c(\rho) \, \int  (\log \psi)^+(|t|) d \mu_f. \]
\end{thm}

\begin{proof} 
\begin{align}
\int  \sum_{n=1}^{\infty} |D_{\rho^{n}}f |^2 dm  &
\le  \sum_{n=1}^{\infty}  \int |\hat{D}_{\rho^{n}}(t)|^2 \, d\mu_f   \nonumber  \\
& = \sum_{n=1}^{\infty}  \int_{|t| \le \varphi({\rho^{n}} ) } |\hat{D}_{\rho^{n}}(t)|^2 \, d\mu_f  + \sum_{n=1}^{\infty}  \int_{|t|> \varphi({\rho^{n}})  }  |\hat{D}_{\rho^{n}}(t)|^2 \, d\mu_f  \nonumber \\
	& = \mbox{I} + \mbox{II} . \label{***}
	\end{align}

 By assumption (a) and Proposition \ref{prop:Main}, 
 there exist $C:\Omega\to \mathbb R^+$, finite P--a.e., such that
 \[\max_{|t|\le \varphi(\rho^n)} |\hat{D}_{\rho^{n}}(t)|^2 \lesssim C(\omega) \frac{1+V_m \log(\varphi(\rho^n))}{b^2_{\rho^n}}.\]
 Therefore, by assumption (b),
 \[ I \lesssim C(\omega) \sum_{n=1}^{\infty}  \frac{1+V_m \log(\varphi(\rho^n))}{b^2_{\rho^n}} \|f\|^2_2
 \lesssim c(\rho) C(\omega) \|f\|^2_2.\]


 For the second term in (\ref{***}), note that $\sup_n \sup_t |\hat{D}_n(t)|$ is bounded. 
 \begin{align*}
\mbox{ II} & = \sum_{n=1}^{\infty} \sum_{k\ge n} \int_{\varphi(\rho^{k}) <|t|\le \varphi(\rho^{(k+1)}) } |\hat{D}_{\rho^n}(t)|^2 d \mu_f 
  \le C \sum_{k=1}^{\infty} k \int_{\varphi(\rho^{k}) <|t|\le \varphi(\rho^{(k+1)}) } d \mu_f \\
  & \le \frac C{ \log_2 \rho} \sum_{k=1}^{\infty}  \int_{\varphi(\rho^{k}) <|t|\le \varphi(\rho^{(k+1)}) }  \log_2 \psi(|t|) d \mu_f 
   \lesssim \frac C{ \log_2 \rho} \int (\log\psi)^{+} |t| d \mu_f.
  \end{align*} 
  Combining the estimates for both terms I and II, the proposition is proven.  \hfill $\square$ 
\end{proof}

\begin{prop} \label{prop:estimate.for.D} 
Let $\{\nu_k\}_{k\in \mathbb N^r}$ be a sequence of independent,  probability transition measures on $\Omega\times \cal B$ such that, for some constant $c>0$,
\[\sum_{{k\in \mathbb N^r}} E\nu_k(|t|>c2^{|k|^{r\beta }})<\infty. \]
Let $K_n$ be defined as in (\ref{eq:Kn}) and $D_n$ the corresponding centered averages. 
Then 
there exists $C(\omega)\in L^1(\Omega)$ such that,
 for any probability space $(X,{\cal D},m)$ with any continuous semi--flow of isometries $\{T_t\}_{t\in \mathbb (R^+)^d}$ acting on it, if $f\in L^2(X)$, 
\[\Bigg\|  \sqrt{ \sum_{n=1}^{\infty} | D_{\rho^n}f |^2 }  \Bigg\|^2_2<  C(\omega) \, c(\rho,\beta) \, \int (\log\log)^{+} |t| d \mu_f. \]

\end{prop}
\begin{proof}
Since $\{\nu_k\}_{k\in \mathbb N^r}$ are probability transition measures, $b_m\sim m^r$ and $V_m\sim m^r$.  $\varphi(x)=c\ 2^{x^{r\beta}}$, for $x\ge 0$, has inverse $(\log_2 (y/c))^{1/(r\beta)}$, $y\ge 1$. It suffices to check the assumptions of Theorem \ref{thm:estimate.for.D.1}.
Assumption (b)  is immediately satisfied for $\beta\in (0,1)$.
\[ \sum_{n\ge 1} \frac{ 1+V_{\rho^n}  \log( \varphi(\rho^{ n}))  }{b_{\rho^n}^2} \sim
\sum_{n\ge 1} \frac{ {\rho^{rn}}  \log( 2^{\rho^{r\beta n}}  )}{\rho^{2rn}} = \sum_{n\ge 1} \frac{1}{  \rho^{(1-\beta) rn}  } < \infty\]
for all $\rho>1$. And assumption (a) becomes the assumption of this proposition.
\hfill $\square$ 
\end{proof}
  
\begin{example} \label{ex:estimate.for.D} Given a sequence $\{\theta_k\}_{k\in \mathbb N^r}$ of independent probability transition measures and $\{n_k\}_{k\in \mathbb N^r}\subset (\mathbb R^+)^d$. 
If 
\begin{align*}
\mbox{(a) } & \quad  |n_k|=O(2^{|k|^{r\beta}})  \mbox{ and }\\
\mbox{(b) } & \quad \mbox{for some constant } c\ge 1,
\sum_{{k\in \mathbb N^r}} E\theta_k(|t|>c \ 2^{|k|^{r\beta}})<\infty.
\end{align*} 
Then the sequence $\nu_k=\theta_k*\delta_{n_k}$ satisfies the condition of Proposition \ref{prop:estimate.for.D}. In particular, if $\{\theta_k\}_{k\in \mathbb N^r}$ are independent copies of one probability transition measure $\theta$, condition (b) becomes
$\sum_{{n\in \mathbb N}} n^{r-1} E\theta(|t|>c \ 2^{n^{r\beta}})<\infty$.
\end{example}

After the square function result along exponential subsequences $\{\rho^n\}$ is obtained, Theorem \ref{thm:estimate.for.D.1} or Proposition  \ref{prop:estimate.for.D}, one is tempted to prove a variational inequality for averages with kernels $\hat{K}_n$. But without additional control on the decay of the kernels, this approach fails to be fruitful. However, this result is enough for universal pointwise convergence of the averages. The steps for proving convergence were inspired by \cite{WirdlRos} and related works.

\begin{defi} Given a set of real numbers $\{x_n\}_{n\in I}$,
where $I$ is a countable index set, define its variation $s$--norm by
$\|x_n\|_{v(s)}= \sup \big(\sum_{j=1}^{\infty} |x_{n_j} - x_{n_{j+1}}|^s\big)^{1/s}$
where the supremum is taken over all possible sub sequences $\{n_j\}$ in $I$.
\end{defi}

\begin{prop}\label{prop:properties.var}(Properties of variation norms)\begin{enumerate}
\item For each $1\le s< \infty$, $\| . \|_{v(s)}$ is a semi-norm.
\item $\|x_n\|_{v(s)} \le 2 \big(\sum_{n=1}^{\infty} |x_n|^s \big)^{1/s}$.
\item $\|x_n\|_{v(s)} \le 2 \sum_k \|\{x_n : n_k \le n < n_{k+1} \} \|_{v(s)}$
for any sequence $n_k$ such that $x_{n_k} = 0$ for all $k$.
\end{enumerate}
\end{prop}
A proof of these properties can be found in \cite{JKRW} along with some discussion
of the applications of the variation norm to ergodic theory.

\begin{prop} \label{prop:conv} With the same assumptions and notation as in Theorem \ref{thm:estimate.for.D.1}, 
if for any $\rho>1$, $\| \|E_{\rho^n}f\|_{v(s)} \|_2 \le C \|f\|_2$ ($s>2$), then there is a universal set $\Omega' \subset \Omega$ of probability 1 such that, for any $\omega \in \Omega'$, 
\begin{enumerate}
\item [(a)]
for any probability space $(X,{\cal D},m)$ with any continuous semi--flow of isometries $\{T_t\}_{t\in \mathbb (R^+)^d}$ acting on it,
 $\lim_{n\to\infty} K^{\omega}_{\rho^n}f(x)$ exists $m$--almost everywhere for any $f\in L^2(X)$ such that $\int  (\log\psi)^{+} |t| d \mu_f <\infty$.
 \item[(b)]
 If in addition, for any $\rho>1$, (a)  $\lim_{\rho\mapsto 1} \sup_n b_{\rho^{n+1}}/b_{\rho^n}=1$,
 for any probability space $(X,{\cal D},m)$ with any continuous semi--flow of positive isometries $\{T_t\}_{t\in \mathbb (R^+)^d}$ acting on it,
 the averages $K^{\omega}_nf$ converge almost everywhere for any $f\in L^2(X)$, such that $\int (\log\psi)^{+} |t| d \mu_f<\infty$.
 \end{enumerate}
\end{prop}

\begin{proof} $\mbox{ }$\\
{\sc\bf Part (a):} 
By Theorem \ref{thm:estimate.for.D.1}, there exists a set
 $\Omega' \subset \Omega$ of probability 1 such that, for any $\omega \in \Omega'$, there is a constant $M=M(\omega,\rho,\varphi)$ such that 
\[  \Big\| \sqrt{\sum_n |D_{\rho^n} f|^2} \Big\|_2 \le M  \Big[\int (\log\psi)^{+} |t| d \mu_f \Big]^{1/2}.\]
Let  $\omega\in \Omega'$. Then  
\begin{align*}
\| \|K_{\rho^n}f\|_{v(s)} \|_2  & \le \| \|E_{\rho^n}f\|_{v(s)} \|_2  + \| \|D_{\rho^n} f\|_{v(s)} \|_2 \\
& \le C \|f\|_2 + 2 \Big\| \sqrt{\sum_n |D_{\rho^n} f|^2} \Big\|_2 \\
& \le (C+2M) \Big[\int (\log\psi)^{+} |t| d \mu_f \Big]^{1/2}.
\end{align*}
Suppose that for some $\omega \in \Omega'$, there exists $f\in L^2(X)$ with
$\int (\log\psi)^{+} |t| d \mu_f <\infty$, for which
 the limit does not exists. Then there exist positive numbers $a$ and $b$ and a set $E\subset X$, with $m(E)>b>0$, such that, for all $x\in E$, 
\[\left|\limsup_j K_{\rho^{n_j}}f(x)	- \liminf_j K_{\rho^{n_j}}f(x)\right|	>a>0.\]
Then, for any $x\in E$, there is an increasing sequence $n_j(x)$ such that
\[\left| K_{\rho^{n_{2j}(x)}}f(x)-K_{\rho^{n_{2j-1}(x)}}f(x)\right|>a, \quad \mbox{for all}\ j>0.\]
Thus
\begin{align*}
a b J^{1/s} < & \int_E \Big[\sum_{j=1}^{2J}|K_{\rho^{n_{j}(x)}}f(x)-K_{\rho^{n_{j+1}(x)}}f(x)|^s\Big]^{1/s} dm\\
\le &   \int \sup_{n_j} \Big[\sum_{j=1}^{2J}|K_{\rho^{n_{j}}}f(x)-K_{\rho^{n_{j+1}}}f(x)|^s \Big]^{1/s} dm\\
\le & \Big\| \|K_{\rho^n}f\|_{v(s)} \Big\|^2_2 \le (C+2M)^2 \int (\log\psi)^{+} |t| d \mu_f . \end{align*}
Letting $J\mapsto \infty$,  we obtain a contradiction. 

\noindent {\sc\bf Part (b)}
Consider a sequence of $\rho_k$ decreasing to 1. By part (a), there exists sets $\Omega_k\subset \Omega$ of probability 1 such that for any $\omega\in\Omega_k$,  $K_{\rho_k^n}f$ converges almost everywhere for any $f\in L^2(X)$ with $\int (\log\psi)^+|t| d\mu_f<\infty$. Let $\Omega'=\cap_k \Omega_k$.Then, for any $\omega\in\Omega'$, for $f\ge 0$, a.e.\ convergence of the full sequence $K^{\omega}_mf$ follows because the $T_t$'s are possitive operators and, for $\rho_k^n\le m<\rho_k^{n+1}$, we have 
\[ \left(\frac{b_{\rho_k^{n}}}{b_{\rho_k^{n+1}}}\right)^d  K^{\omega}_{\rho_k^n}f(x) \le K^{\omega}_mf(x)\le \left(\frac{b_{\rho_k^{n+1}}}{b_{\rho_k^{n}}}\right)^d K^{\omega}_{\rho_k^{n+1}}f(x).\]
The general case follows.   
\hfill $\square$ %
\end{proof}

Theorem \ref{thm:1} follows as a corollary of Proposition \ref{prop:estimate.for.D} and \ref{prop:conv}.

\begin{example} 
Let $\{n_k\}_{k\in\mathbb N}$ be an integer valued sequence such that $n_k=O(2^{k^{\beta}})$ ($0<\beta<1$), and such that 
 $\Big\| \| A_{\rho^n}f\|_{v(s)} \Big\|_2 < C \|f\|_2$ ($s>2$).  
 Let $\{\delta_k\}$ be i.i.d. real--valued random variables such that 
 \[ \sum_{k=1}^{\infty} P(\delta_1>2^{k^{\beta}})<\infty.\]  
 Using Proposition \ref{prop:conv} with $\nu_k=\delta_{n_k+\delta_k}$ or $\nu_k=\delta_{\delta_k}$ respectively, 
 there exists a universal set $\Omega'\subset \Omega$ such that, for any $\omega\in\Omega'$, and 
for any probability space $(X,{\cal D},m)$ with any continuous semi--flow of positive isometries $\{T_t\}_{t\in \mathbb (R^+)^d}$ acting on it,
 both averages $G^{\omega}_nf$ and $H^{\omega}_nf$ converge almost everywhere for any $f\in L^2$ such that $\int (\log\log)^{+} |t| d \mu_f<\infty$.
 Similar results holds when we use $\nu_k=\zeta_k*\delta_{n_k+\delta_k}$ where $\zeta_kf=L_{\epsilon_k}f$ defined in  (\ref{eq:Fn}), under mild conditions on the sequence $\{\epsilon_k\}$. However, for these averages we have stronger results.
\end{example}

\subsection{Smoother Averages}

 We now consider the averages $F^{\omega}_nf(x)$ defined in (\ref{eq:Fn}). Their corresponding kernels
\[\hat{F}^{\omega}_n(t)=\frac 1{n^r} \sum_{k\in [1,n]^r} \hat{\zeta}(\epsilon_k.t) e^{i <(n_k+\delta_k),t>}\]  
decompose as $\hat{F}^{\omega}_n(t)=\hat{E}^{\omega}_n(t) + \hat{D}^{\omega}_n(t)$ where 
\[ \hat{E}^{\omega}_n(t)=\varphi(t) \phi(t) \hat{A}_n(t), \quad
\hat{D}^{\omega}_n(t)=\frac 1{n^r} \sum_{k\in [1,n]^r} [e^{i<\delta_k,t>}\hat{\zeta}(\epsilon_k.t) - E(e^{i<\delta_k,t>})E(\hat{\zeta}(\epsilon_k.t))] e^{i<n_k,t>}, \]
 $\varphi(t)=E(e^{i<\delta_k,t>})$, and $\phi(t)=E(\hat{\zeta}(\epsilon_k.t))$.


\begin{prop} Let $\{\delta_k\}_{k\in \mathbb N^r}$ and $\{\epsilon_k\}_{k\in \mathbb N^r}$ be two independent sequence of positive random vectors in $(\mathbb R^+)^d$. Assume that 
\begin{align*}
\mbox{(a)} & \ \{\epsilon_k\}_{k\in \mathbb N^r} \mbox{ are i.i.d. with } E(\min_{1\le j\le d} |\epsilon_{e,j}|^{-\alpha})<\infty, \mbox{  for some }\alpha>0 \mbox{ and }  e\in \mathbb N^r; \\
\mbox{(b)} & \mbox{ for some constant }c\ge 1,\ \sum_{{k\in \mathbb N^r}} P(|\delta_k|>c \ 2^{|k|^{r\beta }})<\infty, \mbox{  for some } 0<\beta<1, \\
\mbox{(c)} & \ \{n_k\}_{k\in \mathbb N^r} \subset {\mathbb R^+}^d, \mbox{ such that } |n_k|=O(2^{|k|^{r\beta}}).
\end{align*}
 Let $\zeta:\mathbb R^d\to\mathbb R$ be positive, integrable with unit integral satisfying 
\[\quad \mbox{(d)} \qquad \sup_t \prod_{j=1}^d \max(1,|t_j|)^{\alpha} | \hat{\zeta}(t)|< \infty.\] 
There exits a positive function  $C(\omega)$, finite P--a.e., such that for 
for any probability space $(X,{\cal D},m)$ with any continuous semi--flow of positive isometries $\{T_t\}_{t\in \mathbb (R^+)^d}$ acting on it, such that if 
$\Big\| \| A_nf \|_{v(s)} \Big\|_2 \le C \|f\|_2$ ($s>2$), then
\begin{enumerate}
\item[a.]  $\Big\| \| F^{\omega}_nf \|_{v(s)} \Big\|_2 \le C(\omega) c(\beta) \|f\|_2, $ for $s>2$, and
\item[b.] $\lim_{n\to \infty} F^{\omega}_nf$ exits $m$--a.e. for all $f\in L^2(X)$.
\end{enumerate}
\end{prop}
\begin{proof}
Fix $e\in \mathbb N^r$. Since 
\[E_nf=E(F^{\omega}_nf)=\int E(\zeta_{\epsilon_e}(t)) A_n(E(f \circ T_{\delta_e+t}))dt\]
their variational inequality ensues,
\[  \Big\| \| E_nf \|_{v(s)} \Big\|_2 \le \int E(\zeta_{\epsilon_e}(t)) \, \Big\| \| A_n f \circ T_{\delta_e+t}\|_{v(s)} \Big\|_2 dt \le 
C  \|f\|_2. \]
It remains to prove the variational inequality for 
$\{D^{\omega}_nf=F^{\omega}_nf-E_nf\}$. \\
First we obtain a variational inequality along a well chosen sub sequence.
For $k\ge 1$, consider $\{a_j\}_{j\in I_k}$ to be the sequence of equally spaced integers such that $2^k\le a_j <2^{k+1}$, $a_{j+1}-a_j \sim 2^k/k^2$, and the smallest element in this sequence is $2^k$. 
Rename this sequence $\{m_l\}=\cup_k \cup_{j\in I_k} a_j$.
\begin{align*}
  \Big\| \| D^{\omega}_{m_l} f \|_{v(s)} \Big\|^2_2    & \le \Big\| \| D^{\omega}_{m_l} f \|_{v(2)} \Big\|^2_2 \le 2 \sum_{l\ge 1} \| D^{\omega}_{m_l} f \|^2_2\\   
& \le 2 \sum_{k\ge 1} \sum_{j\in I_k}  \int_{|t|\le 2^{2^{\beta k}}}  |\hat{D}^{\omega}_{a_j}(t)|^2 \, d\mu_f 
  + 2 \sum_{k\ge 1} \sum_{j\in I_k}  \int_{|t|> 2^{2^{\beta k}}}  |\hat{D}^{\omega}_{a_j}(t)|^2 \, d\mu_f \\
  &= \mbox{I} +  \mbox{II}.
 \end{align*}
 
 The estimate for first term $I$ follows the same argument as Proposition \ref{prop:estimate.for.D}. By Proposition \ref{prop:Main} there exists $C(\omega)$, a positive integrable function on $\Omega$, such that, for $2^k\le a_j<2^{k+1}$,
\[   \max_{|t| \le 2^{2^{\beta k}}  }  |\hat{D}^{\omega}_{a_j}(t)|^2  \lesssim C(\omega) \frac{   \log( 2^{ 2^{r\beta k} }  )}{2^{rk}} \lesssim C(\omega)  \frac{1}{2^{rk(1-\beta)}}.\]
Since $|I_k| \sim k^2$ and $0<\beta<1$,
\begin{align*}
\mbox{I} & \lesssim C(\omega)  \sum_{k=1}^{\infty} k^2  \frac 1{2^{rk\beta}} \int_{|t| \le 2^{2^{\beta k}}  } d\mu_f\\
 & = C(\omega) c(\beta)  \|f\|^2_2.
 \end{align*}
 For the second term, the assumptions on $\zeta$ imply
 \[ |\hat{D}^{\omega}_n(t)| \lesssim \frac 1{ |t|^{\alpha}} \Big[ \frac 1{n^r} \sum_{l\in [1,n]^r} \frac 1{\min_{1\le j\le d} |\epsilon_{l,j}(\omega)|^{\alpha}} +E(\frac 1{\min_{1\le j\le d} |\epsilon_{l,j}|^{\alpha}})\Big].\]
 By assumption (a), there exist $\Omega''\subset\Omega$ of probability 1 such that, for all $\omega\in\Omega''$, $\sup_n  |\hat{D}^{\omega}_n(t)|\le \frac{C_2(\omega)}{ |t|^{\alpha}}$, with $C_2(\omega)<\infty$ P--a.e.. 
 \begin{align*}
\mbox{ II} & = \sum_k \sum_{j\in I_k} \int_{|t|>2^{2^{\beta k}}  } |\hat{D}^{\omega}_{a_j}(t)|^2 d \mu_f \\
 & \le C^2_2(\omega) \sum_k k^2  \frac 1{2^{\alpha 2^{ \beta k}}} \int_{|t|>2^{2^{\beta k}}  } d \mu_f \\
  & \le C^2_2(\omega)c(\beta) \|f\|^2_2. \\
    \end{align*} 
Now let $\tilde{D}^{\omega}_n=D^{\omega}_{m_k}$ if $m_k\le n <m_{k+1}$. Then 
 $ \| D^{\omega}_nf\|_{v(s)} =  \| D^{\omega}_nf-\tilde{D}^{\omega}_nf\|_{v(s)}  + \| \tilde{D}^{\omega}_nf\|_{v(s)}=\mbox{A+B}.$ 
 The second term $B$ is the variation along the subsequence $\{m_l\}$ handled above. For the first term, by Proposition \ref{prop:properties.var},
 \begin{align*}
 \| D^{\omega}_nf-\tilde{D}^{\omega}_nf\|_{v(s)} & \le 2 \Big(\sum_l \| \{ D^{\omega}_nf: m_l\le n < m_{l+1} \} \|^s_{v(s)}\Big)^{1/s} \\
 &  \le  2\Big( \sum_l \| \{ D^{\omega}_nf: m_l\le n < m_{l+1} \} \|^2_{v(1)} \Big)^{1/2};
 \end{align*}
 and for $m_l\in I_k$,
 \begin{align*}
 \| \{ D^{\omega}_nf: m_l \le n < m_{l+1} \} \|_{v(1)} & \le  \sum_{n=m_l}^{m_{l+1}-1} |D^{\omega}_nf-D^{\omega}_{n+1}f| \\
  & \le  (m_{l+1}-m_l)^{1/2} \big( \sum_{n=m_l}^{m_{l+1}-1} |D^{\omega}_nf-D^{\omega}_{n+1}f|^2 \big)^{1/2}\\
  & \le  \big( \frac{2^k}{k^2} \big)^{1/2} \big( \sum_{n=m_l}^{m_{l+1}-1} |D^{\omega}_nf-D^{\omega}_{n+1}f|^2 \big)^{1/2}.
 \end{align*}
 Since $\hat{D}^{\omega}_n$ are bounded, $|\hat{D}^{\omega}_n(t)-\hat{D}^{\omega}_{n+1}(t)|\sim \frac Cn $, thus
 \begin{align*}
 \|  \| D^{\omega}_nf-\tilde{D}^{\omega}_nf\|_{v(s)} \|^2_2 & \le 4 \sum_k \frac{2^k}{k^2} \sum_{j\in I_k} \sum_{n=a_j}^{a_{j+1}-1} \|D^{\omega}_nf-D^{\omega}_{n+1}f\|^2_2 \\
 & \le 4 \sum_k \frac{2^k}{k^2}   \sum_{n=2^k}^{2^{k+1}-1} 
\|D^{\omega}_nf-D^{\omega}_{n+1}f\|^2_2\\ 
  & \le 4 \sum_k \frac{2^k}{k^2}   \sum_{n=2^k}^{2^{k+1}-1} 
	\int |\hat{D}^{\omega}_n(t)-\hat{D}^{\omega}_{n+1}(t)|^2 d\mu_f \\ 
 & \sim  \sum_k \frac{2^k}{k^2} \sum_{n=2^k}^{2^{k+1}-1} \frac 1{n^2} \int  d\mu_f \\
 & \sim  \sum_k \frac 1{k^2}  \|f\|^2_2  = C \|f\|^2_2.
 \end{align*} 
 
 Now that the variation inequality is established, pointwise convergence for $f\in L^2(X)$ follows by an argument similar to Proposition \ref{prop:conv}. \hfill $\square$ 
 \end{proof}
 
The variational inequality of $\{F_nf\}$ for $f$ in $L^2$ was possible due to the decay of the kernel $\hat{F}_n(t)$. Such decay also yields the a.e. convergence of the related series. This result requires Moricz's \cite{Moricz} estimates for moments of sums of random variables and its extension in Cohen and Cuny \cite{Cohen.Cuny3}.

 \begin{prop} \cite{Cohen.Cuny3} \label{prop:Mor} Let $(Y,\cal C,\mu)$ a probaility space and $\{G_n\}\subset L^2(Y)$. Let $\{\alpha_n\}$ a sequence of non--negative numbers, $\gamma>0$, $C>0$ constants, and $\{A_n\}$ a non-decreasing sequence with growth condition $A_n\lesssim n^{\gamma}$ such that, for all $m>n\ge 1$, 
 \[\left \| \left | \sum_{k=n+1}^m G_k \right |^2  \right \|_2^2 \le A_m \sum_{k=n+1}^m\alpha_k.\]
 If $\sum_{n\ge 1} \alpha_n A_n (\log n)^2 <\infty$, then
 $\sum_{n\ge 1}  G_n$ converges a.e. in $L^2(Y)$, and
 \[ \left \| \sup_{n\ge 1}  \left | \sum_{k=1}^n G_k \right | \right \|_2^2 \le C(\gamma) \sum_{n\ge 1} \alpha_n A_n (\log n)^2.\]
  \end{prop}

 \begin{prop} Let $\{\delta_k\}$ and $\{\epsilon_k\}$ be two independent sequence of i.i.d positive random variables with $E(\min_{1\le j \le d} \epsilon_{e,j}^{-\alpha})<\infty$, for some $\alpha>0$, $e \in \mathbb N^r$. Let $\zeta$ be positive, integrable with unit integral satisfying 
$\sup_t \prod_{1\le j\le d} \max (1,|t_j|^{\alpha}) | \hat{\zeta}(t)|<\infty $. 
Let $\{n_k\}_{k\in \mathbb N^r} \subset (\mathbb R^+)^d$ 
such that $n_k=O(2^{|k|^{r\beta}})$ for some $0<\beta<1$, then there exits a universal set $\Omega'\subset \Omega$ of probability 1 such that for any $\omega\in \Omega'$, for any probability space $(X,{\cal D},m)$ with any continuous semi--flow of isometries $\{T_t\}_{t\in \mathbb (R^+)^d}$ acting on it, the partial sums
\[ \sum_{k\in [1,n]^r} \frac{T_{n_k+\delta_k}L_{\epsilon_k(\omega)}f(x)-T_{n_k}E(T_{\delta_k}L_{\epsilon_k}f(x))}{|k|^r} \mbox{ converges for a.e. } x.\]
Moreover, there is a function  $C:\Omega\to \mathbb R^+$ finite P--a.e. such that 
\[ \Big \| \sum_{k\in \mathbb N^r} \frac{T_{n_k+\delta_k}L_{\epsilon_k(\omega)}f(x)-T_{n_k}E(T_{\delta_k}L_{\epsilon_k}f(x))}{|k|^r}\Big \|^2<C(w) \|f\|_2^2. \]
 \end{prop}

\begin{proof}
$I_{n,m}=\{k\in \mathbb N^r: n<|k| \le m\}$. 
\[
 \left \| \sum_{k\in I_{n,m}}  \frac{T_{n_k+\delta_k}L_{\epsilon_k(\omega)}f(x)-T_{n_k}E(T_{\delta_k}L_{\epsilon_k}f(x))}{|k|^r}\right \|_2^2 \qquad \]
\[ \qquad \le 
  \int \left | \sum_{k\in I_{n,m}} \frac{[e^{i<\delta_k,t>}\hat{\zeta}(\epsilon_k t) - E(e^{i<\delta_k,t>})E(\hat{\zeta}(\epsilon_k t))] e^{in_kt} }{|k|^r} \right |^2 \ d\mu_f.
  \]
By Proposition \ref{prop:Main}, there exists $C_1(\omega)>0$ in $L^1(\Omega)$ such that, for all $|t|\le 2^{m^{r\beta}} $,
\begin{align*}
\left | \sum_{k\in I_{n,m}} \frac{[e^{i<\delta_k,t>}\hat{\zeta}(\epsilon_k.t) - E(e^{i<\delta_k,t>})E(\hat{\zeta}(\epsilon_k.t))] e^{in_kt} }{|k|^r} \right |^2  &
<  C_1(\omega) \sum_{j=n+1}^m \frac 1{j^{r+1}}  \log( 2^{m^{r\beta}}) \\
& \sim C_1(\omega) m^{r\beta} \left (\sum_{j=n+1}^m \frac 1{j^{r+1}}\right ). 
\end{align*}
On the other hand, for $|t|> 2^{m^{r\beta}}$,
\[
 \left | \sum_{k\in I_{n,m}} \frac{[e^{i<\delta_k,t>}\hat{\zeta}(\epsilon_k . t) - E(e^{i<\delta_k,t>})E(\hat{\zeta}(\epsilon_k . t))] e^{i<n_k,t>} } {|k|^r} \right |^2 
\qquad \qquad \]
\begin{align*}
\lesssim 
& \frac{1}{|t|^{2\alpha}} \left [ \sum_{k\in I_{n,m}} \frac 1{|k|^r} \left[
\frac 1{ \min_{1\le j\le d}\epsilon_{k,j}^{\alpha}}+  E(\frac 1{ \min_{1\le j\le d}\epsilon_{k,j}^{\alpha}})\right]  \right ]^2  \\
\le & \frac {m^{2r}}{2^{2\alpha m^{r\beta}}} \left [\frac 1{m^r} \sum_{k\in I_{n,m}} \left[ \frac 1{ \min_{1\le j\le d}\epsilon_{k,j}^{\alpha}} +  E(\frac 1{ \min_{1\le j\le d}\epsilon_{k,j}^{\alpha}})\right] \right ]^2.
\end{align*}

Since $\min_{1\le j\le d}\epsilon_{k,j}^{-\alpha} \in L^1(\Omega)$, there exists $C_2(\omega)>0$ finite P--a.e. such that 
\[ \left | \sum_{k\in I_{n,m}} \frac{[e^{i<\delta_k,t>}\hat{\zeta}(\epsilon_k.t) - E(e^{i<\delta_k,t>})E(\hat{\zeta}(\epsilon_k.t))] e^{i<n_k,t>} }{|k|^r} \right |^2 
\lesssim C^2_2(\omega)  \frac {m^{2r}}{2^{2\alpha m^{r\beta}}}.\]
Combining both estimates, there is $C(\omega)>0$ finite P-a.e. such that
\[ \left \| \sum_{k\in I_{n,m}}  \frac{L_{\epsilon_k}f(T_{n_k+\delta_k}x)-E(L_{\epsilon_k}f(T_{n_k+\delta_k}x))}{|k|^r}\right \|_2^2 \lesssim
C(\omega) m^{r\beta} \left (\sum_{k=n+1}^m \frac 1{k^{r+1}} \right) \|f\|_2^2.
 \]
The result follows by Proposition \ref{prop:Mor}.

\hfill $\square$ 
\end{proof}


\begin{thebibliography}{99}

\bibitem{ABDJ}  Akcoglu, M., Bellow, A., del Junco, A, Jones, R.L., {\em Divergence of averages obtained by sampling a flow}, Proc.\ Amer.\ Math.\ Soc., {\bf 118} (1993), 499--505. 

\bibitem{ADL}  Akcoglu, M,  del Junco, A., Lee, W.M.F, {\em A solution to a problem of A. Bellow}, Almost everywhere convergence, II, Evanston, IL, 1989, Academic Press, Boston, MA, 1991, 1--7. 

\bibitem{ABJLRW}   Akcoglu, M.,  Bellow, A., Jones, R.L., Losert, V., Reinhold, K.,  Weirdl, M., {\em The strong sweeping out property for lacunary sequences, for Riemann sums, and related matters}, Ergodic Th.\ \& Dyn.\ Sys., {\bf 16} (2) (1996) , 207--253.

\bibitem{A1} Assani, I., {\em A weighted pointwise ergodic theorem.} Ann.\ Inst. H. Poincar\'e
Probab.\ Statist., {\bf 34} (1998), 13--150.

\bibitem{A2} Assani, I., {\em Wiener--Wintner dynamical systems.} Ergodic Th.\ \& Dyn.\ Sys.,
{\bf 23} (2003), 1637--1654.

\bibitem{A-duality} Assani, I., {\em Duality and the one--sided Hilbert transform}, Chapel Hill Ergodic Theory Workshops, Contem.\ Math.\ {\bf 356} (2004), 81--90.

\bibitem{BBB} Bergelson, V.,  Boshernitzan, M., Bourgain, J., {\em Some results on non--linear recurrence}, Journal d'Analyse Math., {\bf 62} (1994), 29--46. 

\bibitem{BW} Boukhari, F., Weber, M., {\em Almost sure convergence of weighted series of contractions}, Illinois J. Math, {\bf 46} (2002), 1--21.

\bibitem{B3} Bourgain, J., {\em An  approach  to  pointwise  ergodic  theorems},Geometric  Aspects  of  Functional  Analysis (1986/87), Lecture Notes in Math.\  1317, Springer--Verlag, Berlin, 1988, 204--223.

\bibitem{B2} Bourgain, J.,{\em On pointwise ergodic theorems for arithmetic sets}, CRA Sc.\ Paris, {\bf 305}, Ser 1, 397--402, 1987.

\bibitem{B1} Bourgain, J., Furstenberg, H., Katznelson, I,  Orstein, D.,{\em Return time sequences of dynamical systems,} Publications Mathématiques de L'Institut des Hautes Scientifiques,  {\bf  69} (1989), Issue 1,  42--45.

\bibitem{B} Bourgain,J., {\em Almost sure convergence and bounded entropy}, Israel J. Math., {\bf 63} (1988), 79--97.

\bibitem{Bu} D.L. Burkholder,  {\em Sharp inequalities for martingales and stochastic integrals}, Asterisque,
{\bf 157--158} (1988), 75--94.

\bibitem{Cohen} Cohen, G, {\em On random Fourier-Stieltjes transforms}, Contemp.\ Math., {\bf 430} (2007), 73--88.

\bibitem{Cohen.Cuny2} Cohen, G.,  Cuny, C.,  {\em On random almost periodic series and random ergodic theory}, Ergod.\ Th.\ \& Dyn.\ Sys., {\bf  26} (2006), no. 3, 683--709.

\bibitem{Cohen.Cuny3} Cohen, G.,  Cuny, C.,  {\em On random almost periodic trigonometric polynomials and applications to ergodic theory}, The Annals of Probability, {\bf 34} (2006), no. 1, 39--79.

\bibitem{CohenLin} Cohen, G., Lin, L.,  {\em Extensions of the Menchoff--Rademacher theorem with applications to ergodic theory.} Israel J. Math., {\bf 148}  (2005), no 1,  41--86

\bibitem{CuzickLai} Cuzick, J., Lai, T.L.,  {\em On random Fourier series}, Trans.\ Amer.\ Math.\ Soc., {\bf 261} (1980), 53--80.

\bibitem{Demeter} 
Demeter, D., Lacey, M., Tao, T., Thiele, C., {\em Breaking the duality in the return times theorem}. Duke Math. J., {\bf 143} (2008), no. 2, 281--355.

\bibitem{FanSchneider} 
Fan, A., Schneider, D.,  {\em Sur une in\'egalit\'e de Littlewood--Salem.} Ann.\ Inst.\ H. Poincar\'e Probab.\ Statist., {\bf 39} (2003), 193--216. 
\bibitem{JKRW}
Jones, R.L., Kaufman, R., Rosenblatt, J., Wierdl, W., {\em Ocsillation in ergodic theory}, Ergod.\ Th.\ \& Dynam.\ Sys., {\bf 18}, (1998), no.4, 889--936.

\bibitem{MZ} Marcinkiewicz, J., Zygmund. A., {\em Sur les foncions independantes}. Fund.\ Math.,  {\bf 28} (1937), 60--90.

\bibitem{Moricz} M\'oricz, F., {\em Moment inequalities and the strong law of large numbers}, Z. Wahrsch. Verw. Gebiete, {\bf 35} (1976), 299--314.

\bibitem{PZ} Paley, R., Zygmund, A., {\em On some series of functions, I; II; III.} 
Proc.\ Cambridge Philos.\ Soc., {\bf 26} (1930), 337-??-357; {\bf 26} (1930), 458--??474; {\bf 28} (1932), 190--205.

\bibitem{KR} Reinhold, K., {\em A smoother ergodic average}, Illinois J., {\bf 44} (2000), no. 4, 843--859. 

\bibitem{RHY}
Ren, Y.F., Han--Ying, L., {\em On the best constant in Marcinkiewicz--Zygmund inequality,} Stat.\ \& Prob.\ Letters, (2001) {\bf 53}, no. 3,  227--233.

\bibitem{RN} F. Riesz, B. Sz-Nagy. (1990). Functional analysis, Dover, New York.

\bibitem{WirdlRos} Rosenblatt, J., Wierdl, M., {\em Pointwise ergodic theorems via harmonic analysis} Proc.\ Conference on Ergodic Theory, Alexandria, Egypt,  London Math.\ Society Lecture Notes  205,1993, 3--151.

\bibitem{Ru} Rudin, W, {\em Fourier analysis on groups}. New York: John Wiley \& Sons, 1962; Wiley Classics Library Edition, 1990. 

\bibitem{SZ} Salem, R., Zygmund, A.,  {\em Some properties of trigonometric series whose
terms have random signs,} Acta Math.,  {\bf 91} (1954), 245--301. 

\bibitem{Sch1} Schneider, D., {\em Convergence presque sure de moyennes ergodiques perturb\'ees}, C. R. Acad.\ Sci. Paris, Serie I, {\bf 319} (1994), no. 11, 1201--1206.

\bibitem{Sch2} Schneider, D.,  {\em Theoremes ergodiques perturb\'ees}, Israel J. Math., {\bf 101} (1997), 157--178.

\bibitem{Weber} Weber, M.,  {\em Estimating  random  polynomials  by  means  of metric  entropy  methods,}
Math.\ Inequal.\  Appl., {\bf 3} (2000), 443--457.

\bibitem{Weirdl} Weirdl, M., {\em Pointwise ergodic theorem along the prime numbers}, Israel J. Math., {\bf 64} (1988), no. 3, 315--336.

\end{thebibliography}
\end{document}